\documentclass{amsart}

\usepackage{amsfonts,amssymb,amsmath,amsthm,url}

\def\tfr#1/#2{{#1}/{#2}}				
\def\dfr#1/#2{\displaystyle\frac{#1}{#2}}	
\def\fr#1/#2{\frac{#1}{#2}}				

\def\e#1{e^{#1}}						

\def\Real{\mathbb{R}}

\DeclareMathOperator{\cis}{cis}

\newtheorem{theorem}{Theorem}[section]
\newtheorem{lemma}[theorem]{Lemma}

\author{Jacques G\'elinas}

\address{Ottawa, Canada}

\email{jacquesg00@hotmail.com}

\thanks{This work was done in 2019 while the author was a retired mathematician}

\begin{document}




\title[Blissard's trigonometric series]{Blissard's trigonometric series with closed-form sums}

\begin{abstract}

This is a summary and verification of an elementary note written by John Blissard in 1862 for the Messenger of Mathematics. A general method of discovering trigonometric series having a closed-form sum is explained and illustrated with examples. We complete some statements and add references, using the summation symbol and Blissard's own (umbral) representative notation for a more concise presentation than the original. More examples are also provided.

\end{abstract}

\maketitle


\section{Historical examples}
	Blissard's well structured note \cite{Blissard:1862m} starts by recalling four trigonometric series which ``mathematical writers have exhibited as results of the differential and integral calculus" (A, B, C, G in the table below). Many such formulas had indeed been worked out by Daniel Bernoulli, Euler and Fourier. More examples can be found in 19th century textbooks and articles on calculus, in 20th century treatises on infinite series \cite{Schwatt:1924, Bromwich:1926, Knopp:1954, Hardy:1949} or on Fourier series \cite{Carslaw:1921, Tolstov:1962}, and in mathematical tables \cite{Jolley:1961, SpiegelHandbook:1968, Hansen:1975, BrychkovMarichevPrudnikov:1998}.

	G.H. Hardy motivated the derivation of some simple formulas as follows \cite[p. 2]{Hardy:1949}. We can first agree that the sum of a geometric series $1+x+x^2+\ldots$ with ratio $x$ is $s=\tfr 1/(1-x)$ because this is true when the series converges for $|x|<1$, and ``it would be very inconvenient if the formula varied in different cases"; moreover, ``we should expect the sum $s$ to satisfy the equation $s=1+sx$". With $x=\e{i\theta}$, we obtain immediately a number of trigonometric series by separating the real and imaginary parts, by setting $\theta=0$, by differentiating, or by integrating \cite[\S13 ]{Euler:1783}. The classical theory of Fourier series \cite{Carslaw:1921} can next give a rigorous proof of our conjectured formulas with their domain of convergence, and can also give more sums.

	Abel's limit theorem \cite[p. 154, 405]{Knopp:1954} provides a very simple proof in some cases. Indeed, Abel's summation by parts shows that if a real power series $f(r)=\sum a_k r^k$ with radius of convergence equal to $1$ converges for $r=1$, then $f(r)$ is continuous from the left at $r=1$.  It is thus sufficient to find a closed-form expression for $f(r)$ and take its limit as $r\to1^-$ in order to justify (or discover) a closed-form formula for $f(1)$. For example, if $0<r<1$ and $0<\theta<2\pi$, then
 \cite[p. 211]{Schwatt:1924}
\begin{align*}
	\sum_{n=1}^\infty \fr \left(re^{i\theta}\right)^n/n &= \int_0^{re^{i\theta}} \fr dt / {1-t}
			= -\log\left(1-re^{i\theta}\right) 
			= -\log\left(1-r\cos\theta - i r\sin\theta\right)
\\
		& = -\fr 1/2\log\left((1-r)^2+4r\sin^2\fr\theta/2\right) + i \arctan\left(\fr r\sin\theta/{1-r\cos\theta}\right).
\end{align*}
 Now Abel's test of convergence for power series with positive coefficients decreasing to $0$, proven by
$(1-z)\sum_{k=0}^n a_kz^k=a_0+a_nz^{n+1}+\sum_{k=1}^n(a_k-a_{k+1})z^k$,  shows that $\sum \tfr e^{in\theta}/n$
 converges if $e^{i\theta}\ne 1$ \cite[p. 244]{Bromwich:1926}. Abel's limit theorem can thus be applied here, and by taking real and imaginary parts in Euler's formula for $e^{in\theta}$ we fully justifiy the first two equations (A, B) in the following table, and the next six formulas (C,D,E,F,G,H) follow immediately. Recent derivations of (B) can
be found in \cite{Marin:2012} (Fourier series) and in \cite{Efthimiou:2006} (Laplace transforms).

\renewcommand{\arraystretch}{2.}

$$
\begin{array}{|l|l|c|c|l|} \hline
	& \text{Equation}	& \text{Domain}	& \text{Method}
\\[0.0em] \hline
\text{A} & \displaystyle\sum_{n=1}^\infty \dfr {\cos n\theta}/n = -\log\left(2\sin\dfr\theta/2\right) & 0<\theta<2\pi
				 &	\displaystyle\sum_{n=1}^\infty \dfr \left(re^{i\theta}\right)^n/n, |r|\le1	 	\phantom{\bigg[}
\\[0.5em] \hline
\text{B} & \displaystyle\sum_{n=1}^\infty \dfr \sin n\theta/n = \dfr \pi-\theta/2 & 0<\theta<2\pi
				 & \text{\cite[p. 211]{Schwatt:1924}}
\\[0.5em] \hline
\text{C} & \displaystyle\sum_{n=1}^\infty (-1)^{n-1}\dfr \cos n\theta/n = \log\left(2\cos\dfr\theta/2\right) & -\pi<\theta<\pi
				 & \theta\to \pi-\theta \text{ in A }
\\[0.5em] \hline
\text{D} & \displaystyle\sum_{n=1}^\infty (-1)^{n-1}\dfr \sin n\theta/n = \dfr \theta/2 & -\pi<\theta<\pi
				 & \theta\to \pi-\theta \text{ in B }
\\[0.5em] \hline
\text{E} & \displaystyle\sum_{n=0}^\infty \dfr \cos (2n+1)\theta/{2n+1} = -\dfr 1/2\log\left(\tan\dfr\theta/2\right) & 0<\theta<\pi
				 & \text{ A $+$ C }
\\[0.5em] \hline
\text{F} & \displaystyle\sum_{n=0}^\infty \dfr \sin (2n+1)\theta/{2n+1} = \dfr \pi/4 & 0<\theta<\pi
				 & \text{ B $+$ D }
\\[0.5em] \hline
\text{G} & \displaystyle\sum_{n=0}^\infty (-1)^n\dfr \cos (2n+1)\theta/{2n+1} = \dfr \pi/4 & -\dfr\pi/2<\theta<\dfr\pi/2
				 & \theta\to \dfr\pi/2-\theta \text{ in F }
\\[0.5em] \hline  
\text{H} & \displaystyle\sum_{n=0}^\infty (-1)^n\dfr \sin (2n+1)\theta/{2n+1} = 
							\dfr 1/4 \log \dfr 1+\sin\theta/{1-\sin\theta} & -\dfr\pi/2<\theta<\dfr\pi/2
				 & \theta\to \dfr\pi/2-\theta \text{ in E } 
\\[0.5em] \hline
\text{I} & \displaystyle\sum_{n=1}^\infty \dfr \cos 2n\pi\theta/{(2n\pi)^{2k}} = 
					\dfr (-1)^{k-1}/2 \dfr B_{2k}(\theta) /{(2k)!}  		&	0\le\theta \le1; k=1,2,\ldots
				 & \text{\cite[p. 256]{Bromwich:1908}} , \text{\cite[p. 370]{Bromwich:1926}  } 
\\[0.5em] \hline
\text{J} & \displaystyle\sum_{n=1}^\infty \dfr \sin 2n\pi\theta/{(2n\pi)^{2k+1}} = 
					\dfr (-1)^{k-1}/2 \dfr B_{2k+1}(\theta)/{(2k+1)!}  		& 	\begin{array}{l}	0<\theta <1;k=0\\0\le\theta \le1;k=1,2,\ldots\end{array}
				 & \text{\cite[p. 256]{Bromwich:1908}} , \text{\cite[p. 370]{Bromwich:1926}} 
\\[0.5em] \hline
\text{K} & \displaystyle\sum_{n=-\infty}^\infty \dfr \cos (n-a)\theta/{n-a} = -\dfr \pi/{\tan(a\pi)}
															 & 0<\theta<2\pi; a\in\Real
				 & \text{\cite[p. 230]{Bromwich:1908}} , \text{\cite[p. 371]{Bromwich:1926}} 
\\[0.5em] \hline
\text{L} & \displaystyle\sum_{n=-\infty}^\infty \dfr \sin (n-a)\theta/{n-a} = \pi & 0<\theta<2\pi; a\in\Real
				 & \text{\cite[p. 230]{Bromwich:1908}} , \text{\cite[p. 371]{Bromwich:1926}} 
\\[0.5em] \hline
\text{M} & \displaystyle\sum_{n=1}^\infty (-1)^{n+1}\dfr \cos n\theta/{n^4} =
		 \dfr 7\pi^4/{720} - \dfr\pi^2\theta^2/{24} + \dfr \theta^4/{48} & -\pi\le\theta\le\pi
				 &  k=2,\,\theta \to \dfr \theta/{2\pi}+\dfr 1/2 \text{ in I }
\\[0.5em] \hline
\text{N} & \displaystyle\sum_{n=1}^\infty (-1)^{n+1}\dfr \sin^2 n\theta/{n^4} = \dfr \pi^2\theta^2/{12} - \dfr \theta^4/6
		 & -\dfr \pi/2\le\theta\le\dfr \pi/2
				 & \sin^2 n\theta = \dfr 1 - \cos 2n\theta /2
\\[0.5em] \hline
\end{array}
 $$

\newpage

\section{Binomial summation formula}

Blissard first proves a summation formula involving a positive integer $n$ and a parameter $m$,
which he also used later to derive a formula involving Bernoulli numbers \cite[p. 56]{Blissard:1864}.

\begin{lemma}							
$$
	\sum_{k=0}^n (-1)^k \binom{n}{k} \fr 1/{m+k} = \fr n!/{\prod_{k=0}^n (m+k)},
		\qquad(m>0,n=0,1,2,\ldots).
 $$
\end{lemma}
\begin{proof}
With Blissard's representative notation, let $R_k=\tfr 1/k$, and downgrade exponents into indices after 
expansion\footnote{$R^k:=R_k$ for $k>0$. We use $R^mR^n = R^{m+n}$ and $R^mC-R^nC=(R^m-R^n)C$ for $m>0, n>0$.}.
It is thus required to prove $R^m(1-R)^n = n! \prod_{k=0}^n R^{m+k}$, which is an identity if $n=0$. 
If we assume that this has been proven for a certain integer $n\ge0$ and all $m>0$, then
\begin{align*}
	 R^m \, (1-R)^{n+1} &= R^m \, (1-R) \,  (1-R)^n = R^m \, (1-R)^n - R^{m+1} \, (1-R)^n
\\
		&= n! \, \prod_{k=0}^{n} R^{m+k} -  n! \, \prod_{k=0}^{n} R^{m+1+k}
		= \left( R^m - R^{m+n+1} \right) \, n! \, \prod_{k=1}^{n} R^{m+k}
\\
		&= (n+1) \, R^m \, R^{m+n+1} \, n! \, \prod_{k=1}^{n} R^{m+k}
			\qquad\left(\fr 1/m - \fr 1/{m+n+1} = \fr n+1/{m(m+n+1)}\right)
\\
		&=  (n+1)! \, \prod_{k=0}^{n+1} R^{m+k}.
\end{align*}
Thus the identity in question holds for the next integer $n+1$ and all $m>0$.
\end{proof}
Since the same term $\tfr1/m$ appears on both sides, the identity can be said to hold also as $m\to0$.
Blissard used the same induction, but less concise notation with a few explicit terms and ``\&c.". He gave a longer proof using polynomials and differentiation in \cite[p. 169]{Blissard:1864w}. Another simple proof, similar to \cite{Walton:1864} minus a trigononetric substitution, uses the recurrence relation of the Beta function~:
\begin{align*}
	\sum_{k=0}^n (-1)^k \binom{n}{k} \fr 1/{m+k} &= \int_0^1x^{m-1}(1-x)^n\,dx =
				\fr x^m/m(1-x)^n\bigg|_0^1+\fr n/m\int_0^1x^m(1-x)^{n-1}\, dx 
\\
	&=\fr n(n-1)/{m(m+1)}\int_0^1x^{m+1}(1-x)^{n-2}\, dx =\ldots = \fr \Gamma(n+1)\Gamma(m)/{\Gamma(m+n+1)}.
\end{align*}
Boole \cite[p. 26]{Boole:1957} obtains this from the $n$-th difference of $1/m=\int_0^\infty e^{-mx}\,dx$, a method of Abel, and
one can also use differences of falling factorial powers \cite[p. 188]{GrahamKnuthPatashnik:1989} or partial fraction decomposition.

\section{Maclaurin logarithm expansion}
A second lemma uses the first one to obtain the expansion of a binomial polynomial times a logarithm via
the Cauchy product of absolutely convergent power series (we omit the proof). The $A_{n,k}$ will appear in the examples presented below.
\begin{lemma} \label{maclog}
	If $n$ is a positive integer and $|x|<1$, then
\begin{align*}
	(1+x)^n\log(1+x) &= \sum_{k=1}^n A_{n,k} x^k + n!\,\sum_{k=1}^\infty (-1)^{k-1} \fr (k-1)!/{(n+k)!} x^{n+k},
\\
\noalign{\text{where}}
	A_{n,k} &:= \sum_{j=1}^k  (-1)^{k-j}\binom{n}{j-1} \fr 1/{k-j+1}.
\end{align*}
\end{lemma}

\section{The general recipe}
	A simple method for discovering trigonometric series with a closed-form sum is to start from the Maclaurin expansion of a known function of $x$ and set $x=\e{i\theta}$, as was done above for the geometric series. Blissard states that this works, ``whatever form $f(x)$ may assume'', for the expressions $f(x) \pm f(x^{-1})$ which ``can be evaluated in terms of trigonometrical functions of $\theta$'' with the help of the lemma below. Of course, these combinations give twice the real and imaginary parts of $f(x)$ if $f$ is real. We again omit the proof, done by a simple verification of each case.
\begin{lemma} \label{xeit}
	If $x=\e{i\theta}$ and $\cis\theta := \cos\theta + i\sin\theta$ , then
\begin{align*}
(1)\qquad&	\log x =i\theta,	\quad \log x^{-1} = -i\theta
\\
(2)\qquad& x^n = \cis(n\theta),	\quad x^{-n} = \cis(-n\theta)		\qquad\qquad\qquad(\text{De Moivre})
\\
(3)\qquad& x^n+x^{-n} = 2\cos n\theta,	\quad x^n-x^{-n} = 2i\sin n\theta
\\
(4)\qquad& \left(1+x\right)^n = x^{\tfr n/2} \left(2\cos \fr \theta/2\right)^n, \quad  \left(1+x^{-1}\right)^n = x^{-\tfr n/2} \left(2 \cos \fr \theta/2\right)^n
\\
(5)\qquad& \left(1-x\right)^n = e^{-\tfr i n \pi/2} x^{\tfr n/2} \left(2\sin \fr \theta/2\right)^n, \quad  (1-x^{-1})^n = e^{\tfr i n \pi/2} x^{-\tfr n/2} \left(2\sin \fr \theta/2\right)^n
\\
(6)\qquad&  e^{- i \tfr n \pi/2} x^{m} + e^{\tfr i n \pi/2} x^{-m}= 2\cos\left(m\theta-n\fr \pi/2\right)
\\
(7)\qquad&  e^{- i \tfr n \pi/2} x^{m} - e^{\tfr i n \pi/2} x^{-m}= 2i\sin\left(m\theta-n\fr \pi/2\right)
\end{align*}

\end{lemma}

\section{First logarithmic function}
	If $f(x)=x^m(1+x)^n\log(1+x)$ where $n$ is a positive integer, then for $x=\e{i\theta}$,
\begin{align*}
	f(x) &= \left(2\cos\fr \theta/2\right)^n x^{m+\tfr n/2} \log\left(x^{\tfr 1/2}2\cos\fr \theta/2\right)
\\
	 &= \sum_{k=1}^n A_{n,k} x^{k+m} + n!\,\sum_{k=1}^\infty (-1)^{k-1} \fr (k-1)!/{(n+k)!} x^{m+n+k}
\end{align*}
with the $A_{n,k}$ defined above in lemma \ref{maclog}.
Separating the real and imaginary parts yields two identities (I,II) from
\begin{align*}
\sum_{k=1}^\infty & (-1)^{k-1} \fr (k-1)!n!/{(n+k)!}  \cis \left(m+n+k\right)\theta =
								- \sum_{k=1}^n A_{n,k}\,\cis \left(k+m\right)\theta
\\
	&\quad +\left(2\cos\fr \theta/2\right)^n \left\{ \log\left(2\cos\fr\theta/2\right)\,
			\cis\left(m+\fr n/2\right)\theta + i\fr\theta/2 \,\cis\left(m+\fr n/2\right)\theta \right\}.
\end{align*}
The domain of validity is given as $-\pi<\theta<\pi$.

\section{Second logarithmic function}
	If $f(x)=x^m(1-x)^n\log(1-x)$ where $n$ is a positive integer, then two other identities (III,IV) are derived by the method used in the previous section from
\begin{align*}
(-1)^{n+1}\sum_{k=1}^\infty & \fr (k-1)!n!/{(n+k)!}  \cis \left(m+n+k\right)\theta =
					  \sum_{k=1}^n A_{n,k} \,\cis \left(k+m\right)\theta
\\
	&\quad+\left(2\sin\fr \theta/2\right)^n \left\{ \log(2\sin\fr\theta/2)\,\cis \left(n\fr\pi/2-\left(m+\fr n/2\right)\theta\right)
				+i\fr\theta/2 \,\cis \left(n\fr\pi/2-\left(m+\fr n/2\right)\theta\right) \right\}
\end{align*}
The domain of validity is given as $-2\pi<\theta<2\pi$.

\section{Special cases}
	The four trigonometric series with general term $\tfr\cos (n\theta)/n,\tfr\sin (n\theta)/n$ and their alternating version, (A, B, C, D) in the table above, are obtained by making $m\to0,n\to0$ in the four previous identities (I,II,III,IV) and these examples, first announced in \cite[p. 281]{Blissard:1861}, justify the word ``generalization"
in the original title chosen by Blissard.

	Two identities (I',II') with domain of validity $-\pi<\theta<\pi$ are next obtained by substituting $m=-n$ in (I,II), or by separating the real and imaginary parts in
\begin{align*}
\sum_{k=1}^\infty & (-1)^{k-1} \fr (k-1)!n!/{(n+k)!}  \cis\,k\theta =
								- \sum_{k=1}^n A_{n,k}\,\cis \left(k-n\right)\theta
\\
	&\quad +\left(2\cos\fr \theta/2\right)^n \left\{ \log\left(2\cos\fr\theta/2\right)\,
			\cis\left(-\fr n/2\right)\theta + i\fr\theta/2 \,\cis\left(-\fr n/2\right)\theta \right\}.
\end{align*}

	Substituting $\theta=0$ in I' yields the sum of a numeric alternating series,
$$
	\sum_{k=1}^\infty (-1)^{k-1} \fr (k-1)!n!/{(n+k)!} =  2^n\log 2 - \sum_{k=1}^n A_{k,n}.
 $$

	Another numeric alternating series is summed by substituting $\theta=\pi/2$ in II',
\begin{align*}
\sum_{k=1}^\infty & (-1)^{k-1} \fr (2k-2)!n!/{(n+2k-1)!} =	 \sum_{k=1}^n A_{n,k}\,\sin \left(\fr (n-k)\pi/2\right)
\\
&\quad + 2^{\tfr n/2}\left\{ \fr \pi/4 \cos\left(\fr n\pi/4\right)
		 - \sin \left(\fr n\pi/4\right) \log\left(2\cos \fr \theta/2\right)\right\}.
\end{align*}

\section{Other functions}
	With $f(x)=\arctan(x)$, Blissard sums four series having general term $\tfr \cos[(2n+1)\theta]/n$, $\tfr \sin[(2n+1)\theta]/n$ and their alternating version, obtaining directly (G, H)  then (E,F) with $\pi/2-\theta$ in the table above. 
	
	The function $f(x)=\arctan(\tfr 1/[x^m(1-x)^n])$ is used to sum two alternating series,
\begin{align*}
	\sum_{k=0}^\infty  \fr (-1)^k/{2k+1} \fr \cos(2k+1)m\theta/{(2\cos\theta)^{(2k+1)n}} &=
		\fr 1/2 \arctan\left\{ \fr 2(2\cos\theta)^m\cos m\theta/{(2\cos\theta)^{2n}-1} \right\},
\\
	\sum_{k=0}^\infty \fr (-1)^k/{2k+1} \fr \sin(2k+1)m\theta/{(2\cos\theta)^{(2k+1)n}} &=
	\fr 1/4 \log  \fr {1 + 2 (2\cos\theta)^n \sin m\theta + (2\cos\theta)^{2n}}/
		{1 - 2 (2\cos\theta)^n \sin m\theta + (2\cos\theta)^{2n}}.
\end{align*}

	The function $f(x)=\log(1+x+x^2)$ generates four new series, after replacing $\theta$ by $2\theta$~:
\begin{align*}
	\sum_{k=1}^\infty  \fr (-1)^{k-1}/k \fr \cos 3k\theta/{(2\cos\theta)^k} &= \log \left(\fr 1+2\cos 2\theta/{2\cos\theta}\right)
\\
	\sum_{k=1}^\infty  \fr (-1)^{k-1}/k \fr \sin 3k\theta/{(2\cos\theta)^k} &= \theta
\\
	\sum_{k=1}^\infty  \fr (-1)^{k-1}/k  (2\cos\theta)^k\cos 3k\theta &= \log \left(1+2\cos 2\theta\right)
\\
	\sum_{k=1}^\infty  \fr (-1)^{k-1}/k  (2\cos\theta)^k\cos 3k\theta &= 2\theta.
\end{align*}

	Next, $f(x)=1+x+\ldots+x^{n+r}$ yields the sum of two other alternating series,
\begin{align*}
	\sum_{k=0}^\infty \fr (-1)^k/k \left(\fr \sin n\theta/{\sin r\theta}\right)^k \cos \left(k(n+r)\theta\right)
		 &= \log \left\{\fr \sin(n+r)\theta/{\sin r\theta}\right\},
\\
	\sum_{k=0}^\infty \fr (-1)^k/k \left(\fr \sin n\theta/{\sin r\theta}\right)^k \sin \left(k(n+r)\theta\right)
		 &= n\theta.
\end{align*}

Finally, Blissard derives with $f(x)=\cos x$ and $f(x)=\sin x$ two ``elegant formulae''~:
\begin{align*}
	\tan(\cos\theta) = \fr  {\displaystyle\sum_{n=0}^\infty (-1)^n \dfr \cos (2n+1)\theta/{(2n+1)!}}/
							    {\displaystyle\sum_{n=0}^\infty (-1)^n \dfr \cos (2n)\theta/{(2n)!}}
		=  \fr  {\displaystyle\sum_{n=1}^\infty (-1)^{n-1} \dfr \sin (2n)\theta/{(2n)!}}/
							   {\displaystyle\sum_{n=0}^\infty (-1)^n \dfr \sin (2n+1)\theta/{(2n+1)!}},
\end{align*}
adding that ``these equations hold quite generally''. Indeed, many readers of the Education Times
\cite[p. 64]{Blissard:1866ET2114} noted that, with $x=\e{i\theta}$, the middle quotient
can be reduced by using De Moivre's formula from lemma \ref{xeit}, to
$$
	\fr {\sin x + \sin x^{-1}}/{\cos x + \cos x^{-1}}
			= \fr {2\sin \fr {x+x^{-1}}/2\cos \fr {x-x^{-1}}/2}/{2\cos \fr {x+x^{-1}}/2\cos \fr {x-x^{-1}}/2}
			= \tan \fr {x+x^{-1}}/2  = \tan(\cos\theta).
 $$
The other ``elegant" identity follows likewise from
$$
	\fr {\cos x^{-1} - \cos x}/{\sin x - \sin x^{-1}}
			= \fr {2\sin \fr {x+x^{-1}}/2\sin \fr {x-x^{-1}}/2}/{2\cos \fr {x+x^{-1}}/2\sin \fr {x-x^{-1}}/2}
			= \tan \fr {x+x^{-1}}/2  = \tan(\cos\theta).
 $$

\section{Heuristic rule for the domain of validity}
	The author proposes a rule giving the domain of the independent variable for which the closed-form sum formulas remain valid, but without proof, illustrating it with three examples of erroneous results.
\begin{quote}	
	If such an infinite trigonometrical series as
	$$
		\sum_{k=0}^\infty (\pm1)^k\fr \cos^m(p+kr)\theta/{(p+kr)^n}
	 $$
	 is capable of being summed\footnote{``An infinite trigonometrical series is said to be summed when its value is expressed in finite terms'' \cite[p. 50]{Blissard:1864m}}
	  in terms of the arc $\theta$, then
	 \begin{enumerate}
	 \item The range of application is $|\theta|<\tfr 2\pi/{mr}$ if all signs are positive
	 \item The range of application is $|\theta|<\tfr \pi/{mr}$ if the signs alternate
	 \item The endpoints are \emph{ordinarily} included within the range of application.
	 \end{enumerate}
\end{quote}

	Blissard adopted later a different point of view, accepting that the sum ``is not represented by a single analytical expression" 
	\cite[\S9.11]{WhittakerWatson:1927}~:
\begin{quote}	
	``Further investigation of this important and interesting subject has however led me to perceive that such series may for their whole period i.e. from all values assigned for the arc from $0$ to $2\pi$ have \emph{several ranges of application} and a distinct summation corresponding to each range." \cite{Blissard:1864m}
\end{quote}
	An ingenious method for discovering these ``ranges" is to use elementary decompositions (such as $4\sin^3\theta = 3\sin\theta-\sin 3\theta$) and differentiation to get simpler series (having terms such as $\cos (2n+1)\theta$) which diverge for very obvious values of the variable $\theta$. The formulas for the original series can then be recovered by integration. This method is explained further in the section ``Discovering the discontinuities in the sum of a trigonometrical series'' of \cite{Bromwich:1926} where it is attributed to Stokes (1847) for the case of Fourier series. A simple example from \cite{Blissard:1864m} is
$$
	\sum_{n=0}^\infty (-1)^n \fr \cos (2n+1)\theta /{(2n+1)^3} =
		\begin{cases}
	\fr    \pi^3/{32}                     - \fr \pi\theta^2/8  &       0     \le \theta \le \fr  \pi/2 \\
	\fr   3\pi^3/{32} - \fr \pi^2\theta/4 + \fr \pi\theta^2/8  &  \fr  \pi/2 \le \theta \le \fr 3\pi/2 \\
	\fr -15\pi^3/{32} + \fr \pi^2\theta/2 - \fr \pi\theta^2/8  &  \fr 3\pi/2 \le \theta \le 2\pi.
		\end{cases}
 $$

\section{More examples}
	Blissard stated ``We can obtain by the above method numerous trigonometrical formulae, some of which appear to be remarkable. I subjoin some examples which the young student, for whom this paper is chiefly intended, may work out for himself". Indeed, it is immediate to find some suitable Maclaurin series, for example in 
	\cite[p. 112]{SpiegelHandbook:1968}.
\begin{align*}
	\fr \log(1+x)/{1+x} &= \sum_{n=1}^\infty (-1)^{n-1} H_n x^n,	\qquad(|x|<1,H_n=1+\fr 1/2+\ldots+\fr 1/n)
\\
	\log \left(\fr x/{\sin x}\right) &= \sum_{n=1}^\infty 2^{2n}\fr |B_{2n}|/{2n} \fr x^{2n}/{(2n)!}, 
				\quad(|x|<\pi, B_{2n}=\text{ Bernoulli number }=1,\fr 1/6,-\fr 1/{30},\fr 1/{42},...)
\\
	\log \sec x &= \sum_{n=1}^\infty  2^{2n}\left(2^{2n}-1\right)\fr |B_{2n}|/{2n} \fr x^{2n}/{(2n)!}, 
				\qquad(|x|<\fr\pi/2, B_{2n}=\text{Bernoulli number})
\end{align*}

	The first series corresponds to $m=0$ and $n=-1$ in the first logarithmic function above and yields,
for $-\pi/2<\theta<\pi/2$,
$$
2\sum_{n=1}^\infty  (-1)^{n-1} H_n \cis n\theta =
	\log\left(2\cos\theta\right) + \theta\tan\theta
	+ i\left[\theta - \tan\theta\,\log\left(2\cos\theta\right)\right].
 $$

\section{Numerical verification}
	The typography in 1860's England was not totally reliable, in particular in the first volumes of mathematical publications -- there are many noticable errors in the Quaterly Journal, for example. A computer algebra system (CAS) can be used to detect easily missing terms, digit inversions, sign errors, or other obvious misprints. We provide in this section a minimal set of instructions for the freely available GP/PARI Calculator\footnote{A browser-based implementation is made available at \url{http://pari.math.u-bordeaux.fr/gp.html}}, from Karim Belabas, Henri Cohen and the PARI Group \cite{BelabasCohen:2014}, but other software such as Maple, Mathematica, Python or even Matlab could be chosen instead. In order to detect possible errors in some of the equations of Blissard, we simply verify for some values of the parameters ($n=0,1,\ldots,N$) that the left-hand side of the identity is close, in some sense, to the right-hand side.

\begin{verbatim}
N  = 8;
DS = 16;        \\ number of terms in a Taylor series from \ps
default(seriesprecision,DS);

cis(t) = exp(I*t);
                \\ Floating-point equality (at r=7/8 of precision)
near(x, y, r=7/8) = if(x==y, 1, exponent(normlp(Pol(x-y)))/exponent(0.) > r);

CHECK(e, msg="") = print( if(!e, "Failed: ", "Passed: "), msg );
CHECKN(fn, vn="(n,m)", r=3/4) =  CHECK(N+1==sum(n=0,N, near(eval(Str("L",fn,vn)),\
                   eval(Str("R",fn,vn)), r) ), Str(fn) );

                \\ Binomial summation formula (page 124)
L124(n,m) = sum(k=0,n, (-1)^k*binomial(n,k)/(m+k) );
R124(n,m) = n! / prod(k=0,n,m+k);
CHECKN( 124 );
                \\	Induction proof of Blissard
LA124(n,m) = L124(n+1,m);
RA124(n,m) = L124(n,m) - L124(n,m+1);
CHECKN( A124 );
LB124(n,m) = R124(n,m) - R124(n,m+1);
RB124(n,m) = R124(n+1,m);
CHECKN( B124 );

                \\ Maclaurin logarithm expansion
Akn(k,n) = sum(j=1,k, (-1)^(k-j)*binomial(n,j-1)/(k+1-j) );
L125(n,x)= sum(k=1,n, Akn(k,n)*x^k ) \
                    + n!*sum(k=1,DS, (-1)^(k-1)*(k-1)!/(n+k)!*x^(n+k));
R125(n,x)= (1+x)^n*log(1+x);
CHECKN( 125, "(n,x)" );

                \\ First logarithmic function
L127(n,m,x) = x^m*L125(n,x);
R127(n,m,x) = x^m*R125(n,x);
for(m=0,N, CHECKN( 127, Str("(n,",m,",x)") ));

                \\ Second logarithmic function
L128(n,m,x) = x^m*L125(n,-x);
R128(n,m,x) = x^m*R125(n,-x);
for(m=0,N, CHECKN( 128, Str("(n,",m,",x)") ));

L_127(m,n,t,K=400) = sum(k=1,K,(-1)^(k-1)*n!*(k-1)!/(n+k)!*cis((m+n+k)*t));
R_127(m,n,t) = - sum(k=1,n, Akn(k,n)*cis((k+m)*t)) \
   + (2*cos(t/2))^n*( log(2*cos(t/2))*cis((m+n/2)*t) + I*t/2*cis((m+n/2)*t) );


\\ \sum_{k=1}^\infty (-1)^{k-1} \fr (k-1)!n!/{(n+k)!} = 2^n\log 2 - \sum_{k=1}^n A_{n,k}
L130(n,K=400) = sum(k=1,K, 1.0 * (-1)^(k-1) * (k-1)!*n!/(n+k)! );
R130(n) = 2^n*log(2) - sum(k=1,n,Akn(k,n));
for(n=4,N, CHECK( near(L130(n,400), R130(n), 1/8),\
                  Str("0!/(n+1)!-1!/(n+2)!..., n = ",n) ) );
for(n=0,N, CHECK( near(L_127(-n,n,0,400), L130(n,400), 1/8),\
                  Str("L_127(-n,n,0) = L130(n), n = ",n) ) );

\\ \sum_{k=1}^\infty (-1)^{k-1} \fr (2k)!n!/{(n+2k+1)!}
LM1_131(n,K=400) = sum(k=1,K,(-1)^(k) * 1.0 * (2*k-2)!*n!/(n+2*k-1)!);
for(n=2,N, CHECK( near(-imag(L_127(-n,n,Pi/2,400)), LM1_131(n,400), 1/8),\
                  Str("L_127(-n,n,Pi/2) = LM1_131(n), n = ",n) ) );
for(n=2,N, CHECK( near(-imag(R_127(-n,n,Pi/2)), LM1_131(n,400), 1/8),\
                  Str("R_127(-n,n,Pi/2) = LM1_131(n), n = ",n) ) );

H(n) = sum(k=1,n,1/k);        \\ More examples

CHECK( log(1+x)/(1+x) == sum(n=1,DS,(-1)^(n-1)*H(n)*x^n),"log(1+x)/(1+x)");
CHECK( log(x/sin(x) ) == sum(n=1,DS/2, m=2*n;\
                  2^m*abs(bernfrac(m))/m*x^m/m!), "log(x/sin x)" );
CHECK( log(1/cos(x) ) == sum(n=1,DS/2, m=2*n;\
                  2^m*abs(bernfrac(m))/m*x^m/m!*(2^m-1)), "log(sec x)" );
\end{verbatim}

\section{Acknowledgement}
The author thanks Mr Garry Herrington for reviewing an earlier version of this document and communicating useful comments, corrections and suggestions to improve it.

\newpage

\bibliographystyle{plain}

\end{document}